\documentclass[12pt,twoside]{article}
\usepackage{amsmath,amsfonts,amssymb,amsthm}
\usepackage{bbm}
\usepackage{graphicx}
\usepackage{color}
\def\XXint#1#2#3{{\setbox0=\hbox{$#1{#2#3}{\int}$ }
\vcenter{\hbox{$#2#3$ }}\kern-.6\wd0}}

\newcommand{\grad}{\nabla}

\renewcommand{\div}{\grad\cdot}
\newcommand{\N}{\mathbf{N}}

\newcommand{\R}{\mathbf{R}}

\DeclareMathOperator{\id}{id}

\DeclareMathOperator{\spt}{supp}
\def\loc{{\mathrm{loc}}}

\newcommand{\D}{\ensuremath{\mathcal{D}}}

\def\opt{{\mathrm{opt}}}

\def\per{{\mathrm{per}}}

\newcommand{\eps}{\varepsilon}

\newtheorem{ex}{Example}
\newtheorem{prop}{Proposition}
\newtheorem{theorem}{Theorem}
\newtheorem{lemma}{Lemma}

\newtheorem{definition}{Definition}

\newcommand{\tacka}{\, \cdot\,}

\begin{document}

\title{A quantitative theory for the continuity equation}
\author{Christian Seis\thanks{Institut f\"ur Angewandte Mathematik, Universit\"at Bonn, Endenicher Allee 60, 53115 Bonn, Germany}}

\maketitle

\begin{abstract}
In this work, we provide stability estimates for the continuity equation with Sobolev vector fields. The results are inferred from  contraction estimates for certain logarithmic Kantorovich--Rubinstein distances. As a by-product, we 
obtain a new proof of uniqueness in the DiPerna--Lions setting. The novelty in the proof lies in the fact that it is not based on the theory of renormalized solutions.
\end{abstract}

\section{Introduction}
When $u: [0,T]\times \R^d\to \R^d$, $f: [0,T]\times \R^d\to \R$ and $\bar \rho: \R^d\to\R$ are smooth functions, the solution of the Cauchy problem for the continuity equation
\begin{equation}
\label{23}
\left\{\begin{array}{rcl}\partial_t \rho + \div\left(u\rho\right) &=& f,\\ \rho(0,\tacka) &=& \bar \rho\end{array}\right.
\end{equation}
is found by the method of characteristics: If we denote by $\phi: [0,T]\times \R^d\to \R^d$ the flow of the vector field $u$, i.e.,
\begin{equation}
\label{16}
\left\{\begin{array}{rcl}\partial_t \phi(t,x) &=&  u(t,\phi(t,x)),\\ \quad \phi(0,x)&=&x,\end{array}\right.
\end{equation}
for any $(t,x)\in [0,T]\times \R^d$, then the solution to \eqref{23} is given by the formula
\begin{equation}
\label{16c}
\rho(t,\phi(t,x)) \det \grad\phi(t,x) = \bar \rho(x) + \int_0^t f(s,\phi(s,x))\det \grad\phi(s,x)\, ds.
\end{equation}
In the non-smooth setting, solutions have to be defined in the sense of distributions. A complete theory of distributional solutions, including existence, uniqueness and stability properties,  is provided in the seminal works of DiPerna and Lions \cite{DiPernaLions89} and Ambrosio \cite{Ambrosio04}.

The approach of DiPerna, Lions and Ambrosio relies on the theory of renormalized solutions. Roughly speaking, renormalized solutions are distributional solutions to which the chain rule applies in the sense that, for every suitable $\beta\in C^1(\R)$, $\beta(\rho)$ solves the continuity equation with source $\beta'(\rho)f +(\div u)(\beta(\rho) - \rho\beta'(\rho))$ and initial datum $\beta(\bar \rho)$. Whether distributional solutions are renormalized solutions depends on the regularity of $u$. It has been proved in the original paper by  DiPerna and Lions \cite{DiPernaLions89} that this is true under the condition that $u\in L^1(W^{1,1})$ and $\div u \in L^1(L^{\infty})$. DiPerna and Lions furthermore show that the latter can not be relaxed, in the sense that there are (stationary) vector fields in $W^{1,p}$ for any $p<\infty$ whose divergence is unbounded and that do not possess this renormalization property. Likewise, the authors construct solutions to the continuity equation with  $u\in W^{s,1}$ for any $s<1$ (and $\div u=0$) that are not renormalized. In \cite{Ambrosio04}, Ambrosio extends DiPerna's and Lions's results to vector fields $u\in L^1(BV)$. A counterexample in the non-$BV$ setting is provided by Depauw \cite{Depauw03}.

The merit of renormalization theory relies on a simple proof of uniqueness and stability. For instance, if $\eta$ denotes the difference of two solutions to the Cauchy problem \eqref{23}, the choice 
$\beta(z)=z^2$ yields
\[
\partial_t \eta^2 + \div\left(u\eta^2\right) = - (\div u)\eta^2,
\]
and thus, integration in space and a Gronwall argument shows that
\[
\|\eta\|_{L^{\infty}(L^2)} \le \|\bar \eta\|_{L^2} \exp^{\frac12}\left( \|\div u\|_{L^1(L^{\infty})}\right).
\]
Thus, if the initial datum $\bar \eta$ is zero, then $\eta$ vanishes identically. For a recent review on the well-posedness theories for the continuity equation \eqref{23} and the related ordinary differential equation \eqref{16}, we refer the reader to the lecture notes \cite{AmbrosioCrippa14}.

Renormalization theory is also powerful as it applies to a fairly broad class of transport or kinetic equations, e.g.,  \cite{DiPernaLions88,DiPernaLions89b,DiPernaLions89c}.
What the theory does not provide are stability estimates and bounds on the mixing or unmixing efficiencies in terms of the regularity of the advecting vector field.
 Such estimates, however, attracted much attention recently. For instance, in  \cite{LinThiffeaultDoering11,Seis13b,IyerKiselevXu14}, the continuity equation is considered as model for  mixing  of tracer particles by a viscous fluid flow. An important question in engineering applications is how well tracers can be mixed under a constraint on the advecting velocity field. Typically, one is interested in optimal mixing rates in terms of the kinetic energy $\|u\|_{L^2}$ or, more importantly, the viscous dissipation $\|\grad u\|_{L^2}$. 
The works \cite{Seis13b,IyerKiselevXu14} provide lower bounds on the rate of  exponential  decay of the $H^{-1}$ norm by $\|\grad u\|_{L^1(L^p)}$. Optimality of these bounds is proved in \cite{AlbertiCrippaMazzucato14,YaoZlatos14}.

The goal of the present work is to establish stability estimates for continuity equations with Sobolev vector fields that allow for variations of vector field, source, and initial datum. We demonstrate the strength of these estimates by providing a new proof of uniqueness of distributional solutions. Opposed to the theory of DiPerna, Lions, and Ambrosio, our approach does not rely on   renormalized solutions. Instead, we obtain uniqueness from a contraction estimate under suitable integrability  assumptions on the solutions.

Our approach is motivated by a related work  by Crippa and De Lellis for the ordinary differential equation \eqref{16}. In \cite{CrippaDeLellis08}, the authors derive simple stability estimates for suitably generalized flows, so-called regular Lagrangian flows, in the case of Sobolev vector fields $u$. These estimates allow for a direct proof of well-posedness.  Prior to the work of Crippa and De Lellis, uniqueness and stability were obtained quite indirectly and exploited the connection to the continuity equation \eqref{23} and the transport equation
\[
\left\{\begin{array}{rcl}\partial_t \rho + u\cdot \grad \rho &=& f,\\ \rho(0,\tacka) &=& \bar \rho\end{array}\right.
\]
via the method of characteristics, cf.\ \cite{DiPernaLions89,Ambrosio04}. Crippa's and De Lellis's approach has been partially extended to the $BV$ setting later by Jabin \cite{Jabin10} and Hauray and Le Bris \cite{HaurayLeBris11}.

Focussing on the case $p>1$, Crippa and De Lellis prove that any two solutions $\phi$ and $\tilde \phi$ of \eqref{16} satisfy estimates of the type
\begin{equation}
\label{16e}
\sup_{t\in(0,T)} \int \log\left(\frac{|\phi(t,x)  - \tilde \phi(t,x)|}{\delta} +1\right) dx \lesssim \|\grad u\|_{L^1(L^p)},
\end{equation}
{\em uniformly} in $\delta>0$. That means, trajectories can only vary in a tube with diameter of order $\delta$. As $\delta\to0$, this tube shrinks to a single curve, which proves uniqueness. Crippa's and De Lellis's logarithmic estimates generalize the well-known estimate valid for flows of Lipschitz vector fields
\[
\sup_{t\in(0,T)} \sup_{x\not=y}  \log\left( \frac{|\phi(t,x) - \phi(t,y)|}{|x-y|}\right) \le \|\grad u\|_{L^1(L^{\infty})}.
\]
The latter states that trajectories diverge at most exponentially  in time and yields continuous dependence on the  initial data. It should be mentioned that the (not displayed) constant in \eqref{16e} depends on the uniform bound on $\div u$. In the Lipschitz case, this bound is redundant.

Our argument for the continuity equation is very similar. Our substitute for the quantity on the left-hand side of \eqref{16e} is a bounded variant of the Kantorovich--Rubinstein distance
\[
\D_{\delta}(\eta) : = \inf_{\pi\in\Pi(\eta_+,\eta_-)} \iint  \log\left(\frac{|x-y|}{\delta}+1\right)\, d\pi(x,y),
\]
where the $\Pi(\eta_+,\eta_-)$ is the set of all joint measures on the product space $\R^d\times \R^d$ with marginals $\eta_+ := \max\{\eta,0\}$ and $\eta_- := \max\{-\eta,0\}$,\footnote{The reader will find a proper definition of Kantorovich--Rubinstein distances in Section \ref{S2} below.} and where $\eta$ denotes the difference of two solutions of the Cauchy problem for the continuity equation \eqref{23}.
Notice that the total mass of $\eta_+$ and $\eta_-$ is the same along the evolution, so that $\Pi(\eta_+,\eta_-)$ is non-empty. We will prove that
\begin{equation}
\label{16d}
\sup_{t\in (0,T)} \D_{\delta}(\eta(t,\tacka)) \lesssim \|\grad u\|_{L^1(L^p)},
\end{equation}
{\em uniformly} in $\delta$, and thus, arguing similarly as for \eqref{16e}, upon choosing $\delta$ arbitrarily small we see that $\eta$ must be zero. In the derivation of \eqref{16d}, we will work directly with the distributional formulation of the continuity equation.

Estimates in the flavor of \eqref{16d} were derived earlier in \cite{BOS,OSS,Seis13b} in the context of demixing and mixing problems, though the proofs in there rather rely on the  Lagrangian framework \eqref{16} via \eqref{16c}. Due to the logarithmic cost function, \eqref{16d} can be considered as a contraction estimate for $\exp(\D_{\delta}(\eta))$. This was the perspective taken in \cite{Seis13b} to derive exponential lower bounds on mixing measures in terms of the generalized viscous dissipation rate $\|\grad u\|_{L^1(L^p)}$. Moreover, building up on \eqref{16d}, the author recently computed optimal convergence rates for numerical schemes (jointly with Schlichting) \cite{SchlichtingSeis16a} and of diffusive perturbations \cite{Seis16b}. Logarithmic energy-type estimates for renormalized solutions of the continuity equation where derived earlier in \cite{ColomboCrippaSpirito15}.

Before rigorously stating our main results, we specify some assumptions on the data and introduce our notion of weak solutions. We let $1\le p,q\le\infty$ be fixed with $1/p+1/q=1$ and consider vector fields $u$ in $L^1_{\loc}(\R;W^{1,p}(\R^d))$. For any locally integrable choice of initial datum $\bar \rho$ and source term $f$, the following definition of distributional solutions is reasonable:

\begin{definition}\label{D1}
A function $\rho\in L^{\infty}_{\loc}(\R;L^q(\R^d))$ is called a distributional solution of the Cauchy problem for the continuity equation \eqref{23} if 
\[
\int_0^{\infty}\int \rho\left(\partial_t \zeta + u\cdot \grad\zeta\right) + f\zeta\, dxdt  + \int \zeta(t=0,\tacka) \bar \rho\, dx=0,
\]
for all functions $\zeta\in C^{\infty}_{c}( [0,\infty)\times \R^d)$.
\end{definition}

Under suitable integrability assumptions on $\bar \rho$ and $f$, existence of distributional solutions is obtained via standard approximation techniques. Indeed, arguing as in the work of DiPerna and Lions \cite{DiPernaLions89}, it is not difficult to see that for smooth and compactly supported data, we have the estimate
\begin{equation}
\label{25a}
 \|\rho\|_{L^{\infty}(L^q)} \le \exp^{1-\frac1q}\left(\|\div u\|_{L^1(L^{\infty})}\right)\left(\|\bar\rho\|_{L^q}+ \|f\|_{L^1(L^q)}\right) .
\end{equation}
This estimate guarantees compactness (the case $p=1$ requires a bit more work, cf.\ \cite[p.\ 515]{DiPernaLions89}),   provided that the right-hand side is finite.  In the present work, we will always assume that distributional solutions exist in the sense of Definition \ref{D1}.

It is clear that \eqref{25a} carries over to renormalized solutions by choosing a smooth approximation of $\beta(z) = |z|^q$. Our ambition is to avoid the theory of renormalized solutions. Still, in our proof of uniqueness we need certain integrability assumptions on $\rho$:

\begin{theorem}[Uniqueness]
\label{T1}
Let $p,\, q \in[1,\infty]$ be given  with $1/p+1/q=1$. Suppose that $u: (0,T)\times \R^d\to\R^d$ is a measurable function with $\grad u\in L^1((0,T);L^p(\R^d))$.  Then there exists at most one distributional solution in $ L^{\infty}((0,T);L^1\cap L^q(\R^d))$ to the Cauchy problem for the continuity equation \eqref{23}.
\end{theorem}

We recall that this result was previously derived by DiPerna and Lions \cite[Corollary II.1]{DiPernaLions89}. Our contribution is a new quantitative proof.

In fact, DiPerna's and Lions's theory goes far beyond the above theorem. The duo establishes well-posedness in a setting where distributional solutions are not even defined \cite[Theorem II.3]{DiPernaLions89}. This is achieved by restricting the class of admissible $\beta$'s to bounded functions. For these, the renormalized equation and, in particular, the transport term $\div (u\beta(\rho))$ make sense distributionally even if $\rho$ is only integrable.

The requirement in Theorem \ref{T1} that distributional solutions belong to $L^1$  can be dropped, if, for instance, $u$ is a bounded vector field. In this case, the proof  has to be modified and a localized version of the continuity equation \eqref{23} has to be considered.
%
% It is furthermore interesting to notice that in the above result it is not assumed that $u$ has bounded divergence. This condition, which is crucial in the DiPerna--Lions theory, is in our setting replaced by the slightly weaker integrability assumption contained in our definition of weak solutions.
% 
Our main result is the following stability estimate in the case $p>1$:

\begin{theorem}[Stability]
\label{T2}
Let $p\in(1,\infty]$ and $q\in [1,\infty)$ be given  with $1/p+1/q=1$.
Let $\rho_1$ and $ \rho_2$ in $ L^{\infty}((0,T);L^1\cap L^q(\R^d))$ be two solutions to the continuity corresponding to the data $(u_1,f_1,\bar \rho_1)$ and $(u_2,f_2,\bar \rho_2)$, respectively. Assume that $u_1\in L^1((0,T);W^{1,p}(\R^d))$ and that
\[
r:= \|u_1-u_2\|_{L^1(L^p)} +\|f_1-f_2\|_{L^1(L^q)} +  \|\bar \rho_1-\bar \rho_2\|_{L^q} \ll 1.
\]
Then
\[
\eta : = \rho_1-\rho_2 - (\bar \rho_1-\bar\rho_2) - \int_0^t(f_1-f_2)\, ds
\]
is bounded in $L^{\infty}((0,T);L^1\cap L^q(\R^d))$ and there exists a constant $C$ independent of $r$ such that
such that
\begin{equation}
\label{25b}
\|\eta\|_{L^{\infty}(W^{-1,1})} \le\frac{C}{|\log r|}.
\end{equation}
\end{theorem}
In the statement, we have used the notation $W^{-1,1}$ for the dual space of $W^{1,\infty}$, endowed with the norm
\[
\|\eta\|_{W^{-1.1}} = \sup\left\{ \int \eta \varphi\, dx:\: \|\varphi\|_{W^{1,\infty}}\le 1\right\}.
\]

Estimates analogous to \eqref{25b} in the Lagrangian setting \eqref{16} can be found in \cite[Theorem 2.9]{CrippaDeLellis08}.

We remark that DiPerna and Lions \cite{DiPernaLions89} prove $L^1_{\loc}$ stability for the continuity equation by the use of renormalized solutions. The new contribution here is the stability estimate.

The following example by De Lellis, Gwiazda and \'Swierczewska-Gwiazda shows that one {\em cannot} expect strong stability estimates for the continuity equation.

\begin{ex}[\cite{DeLellisGwiazdaSwierczewska16}]\label{example}
For $k\in\N$, consider the one-dimensional vector fields $u_k(x) = k^{-1} \sin(kx)$. We denote by $\rho_k$ the solutions of the homogeneous continuity equation
\[
\partial_t\rho_k + \partial_x(u_k\rho_k) = 0,\qquad \rho_k(0)\equiv1.
\] 
It is clear that $u_k\to 0$ uniformly, while $\partial_x u_k\rightharpoonup 0$ in $L^1_{\loc}$. If we denote by $\phi_k$ the associated flow, then both $\phi_k$ and $ \phi_k^{-1}$ converge uniformly to the identity on $\R$, but $\partial_x \phi_k$ does {\em not} converge strongly in $L^1_{\loc}$. Using \eqref{16c} we then compute that on any bounded interval $I$ in $\R$ and for any $T>0$, it holds
\[
\int_0^T \int_I |\rho_k(t,x)  -1|\, dx = \int_0^T\int_{\phi_k^{-1}(t,I)} | 1- \partial_x \phi_k(t,x)|\, dx.
\]
In particular, 
\[
\rho_k\,\, \not\!\!\longrightarrow 1\quad\mbox{in } L^1_{\loc}.
\]
\end{ex}

Our method fails in the case $p=1$ for the same reason why Crippa's and De Lellis's approach for \eqref{16} fails: It is not clear if estimate \eqref{16d} holds true if $u\in L^1(W^{1,1})$ or $u\in L^1(BV)$. It is not difficult to show that \eqref{16d} holds in these cases if one allows for an error of the order $|\log \delta|$, e.g.,
\[
\sup_{t\in(0,T)} \D_{\delta}(\eta(t,\tacka)) \lesssim |\log\delta| \|\grad u\|_{L^1(L^1)}.
\]
This estimate is critical because $\D_{\delta}(\eta)\sim  |\log\delta|$ if $\eta$ varies on a scale of order $1$. In the Sobolev case, following an idea of Jabin \cite{Jabin10}, this error can be lowered to order $o(|\log\delta|)$ but one looses the explicit dependence on $\grad u$. As a consequence, such a bound is still enough for proving uniqueness, but too weak to construct explicit stability estimates. The $p=1$ case of \eqref{16d} is related to an open conjecture of Bressan \cite{Bressan03}.

The mathematical reason why our proof (and the one of Crippa and De Lellis) does not extend to the case $p=1$ in a clean way is connected to failing Calder\'on--Zygmund theory, more precisely, to the fact that the maximal function operator fails to be continuous from $L^1$ to $L^1$. We refer to the discussion on page \pageref{page1} and Example \ref{E1} on page \pageref{E1} for more details.

To the best of our knowledge, in this paper, it is for the first time that uniqueness for the continuity equation with non-smooth vector fields is obtained without the use of renormalization theory, and, more importantly, that explicit stability estimates are derived. Optimal transportation tools were previously used for nonlinear continuity equations, e.g.,  for the Vlasov--Poisson system \cite{Loeper06a} and the 2D Euler vorticity equation \cite{Loeper06b}.

The paper is organized as follows. In the following section, we introduce and discuss Kantorovich--Rubinstein distances. Section \ref{S4} contains the proofs.

\section{Kantorovich--Rubinstein distances}\label{S2}

The goal of this section is to give an overview on some basic results in the theory of optimal transportation with metric cost functions. We choose a presentation that is tailored to our needs, and in particular, we will focus on a rather ``smooth'' setting. For possible generalizations as well as a comprehensive introduction to the topic, we refer to Villani's monograph \cite{Villani03} and the references therein.

Given two (nonnegative) distributions $\eta_1$ and $\eta_2$ on $\R^d$ with same total mass,
\[
\int \eta_1\, dx = \int\eta_2\, dx,
\]
we consider the set of all joint measures $\Pi(\eta_1,\eta_2)$. That is, $\pi\in \Pi(\eta_1,\eta_2)$ is characterized by the requirement that
\begin{equation}\label{1}
\pi[A\times \R^d] = \int_A\eta_1\, dx,\qquad \pi[\R^d\times A]  = \int_A\eta_2\,dx,
\end{equation}
for all measurable sets $A$ in $\R^d$. In the theory of optimal transportation, $d\pi(x,y)$ measures the amount of mass that is transferred from the producer at $x$ to the consumer at $y$. Accordingly, we refer to $\pi$ as a {\em transport plan}. Condition \eqref{1} can  equivalently be stated as
\begin{equation}
\label{2}
\iint \zeta_1(x) +\zeta_2(y)\, d\pi(x,y) = \int \zeta_1\eta_1\, dx + \int \zeta_2\eta_2\, dy,
\end{equation}
for all $\zeta_1 \in L^1(\eta_1\, dx)$ and $\zeta_2 \in L^1(\eta_2\, dx)$. Supposing that transport of mass over a distance $z$ is described by a continuous increasing cost function $c(z)$, the problem of optimal transportation consists of finding a transport plan that minimizes the total transportation cost. The minimal transportation cost is thus
\begin{equation}
\label{4}
\D_c(\eta_1,\eta_2): = \inf_{\pi\in \Pi(\eta_1,\eta_2)} \iint c(|x-y|)\, d\pi(x,y).
\end{equation}
We will always assume that $c$ is bounded on $\R_+$.
Following the direct method of calculus of variations, it is then not hard to see that the infimum is actually attained. 

In  this paper, we will study optimal transportation with  {\em concave cost functions} $c: [0,\infty)\to [0,\infty)$ and $c(0)=0$. These functions induce a metric $d(x,y) = c(|x-y|)$ on $\R^d$, and the optimal transport problem has the dual formulation
\begin{equation}
\label{3}
\D_c(\eta_1,\eta_2) = \sup_{\varphi}\left\{ \int\varphi(\eta_1-\eta_2)\, dx:\: |\varphi(x)-\varphi(y)|\le d(x,y)\right\}.
\end{equation}
The latter is known as the Kantorovich--Rubinstein theorem (cf.\ \cite[Theorem 1.14]{Villani03}). Hence, for concave cost functions, the optimal transportation problem only depends on the difference $\eta_1-\eta_2$, and thus, in this case, the problem generalizes to distributions that are not necessarily nonnegative. We conveniently write
\[
\D_c(\eta)   := \D_c(\eta,0) := \D_c(\eta_+,\eta_-)
\]
for any function $\eta$ in $L^1(\R^d)$ with zero ``mean'',
\[
\int\eta\, dx=0,
\]
and where $\eta_+$ and $\eta_-$ denote the positive and negative part of $\eta$, respectively. Furthermore, because $d(x,y)$ is a bounded metric on $\R^d$, the minimal transportation cost $\D_c(\eta_1,\eta_2)$ defines a metric on $L^1(\R^d)$, which goes by different names depending on the mathematical community. We follow Villani in \cite{Villani03} and refer to $\D_c(\eta_1, \eta_2)$ as {\em Kantorovich--Rubinstein distance}.

The dual problem \eqref{3} admits a maximizer $\varphi_{\opt}$ that saturates the Lipschitz constraint in the form
\[
\varphi_{\opt}(x) - \varphi_{\opt}(y) = d(x,y)\quad\mbox{for $d\pi_{\opt}$-almost all }(x,y),
\]
where $\pi_{\opt}$ is a minimizer of $\D_c(\eta_+,\eta_-)$ in the primal formulation \eqref{4}. We will often refer to $\varphi_{\opt}$ as a {\em Kantorovich--Rubinstein potential}. It is clear that $\varphi_{\opt}$ is non-unique: we can always add a constant to $\varphi_{\opt}$ without changing the expectation with respect to $\eta$. We can therefore always choose $\varphi_{\opt}$ as a bounded $d$-Lipschitz function.

If the cost function $c$ is {\em strictly concave}, the above identity and the Lipschitz constraint in turn imply that Kantorovich--Rubinstein potentials are weakly differentiable in $\spt(\eta)$ with
\begin{equation}
\label{5}
\grad\varphi_{\opt}(x) = \grad\varphi_{\opt}(y) ,\qquad\grad\varphi_{\opt} (x) = \grad_x d(x,y) = c'(|x-y|)\frac{x-y}{|x-y|},
\end{equation}
for $d\pi_{\opt}$-almost all $(x,y)$. Notice that $\pi_{\opt}$ is supported away from the diagonal and thus $x\not=y$ in \eqref{5}.\footnote{The second formula in \eqref{5} can be verified as follows: For $d\pi_{\opt}$-almost all $(x,y)$ it holds that $x\not=y$, and thus, for any $z\in \R^d$ and $s\in \R\setminus\{0\}$ small,
\[
\varphi_{\opt}(x+sz)  -\varphi_{\opt}(x)\le d(x+sz,y) - d(x,y).
\]
Hence, dividing by $s$ be find \eqref{5} as $s\to 0$.}

\section{Proofs}\label{S4}

Throughout this section, $\eta$ will always be an integrable, mean-zero distributional solution to a continuity equation.
In order to measure the distance of such a solution $\eta$ to the trivial solution, we design a Kantorovich--Rubinstein distance with {\em bounded logarithmic cost}, more precisely, for any positive $\delta$ and $R$, we set
\[
c_{\delta,R}(z) = \begin{cases} \log\left(\frac{z}{\delta}+1\right) & \mbox{for }z\le R,\\ 
\log\left(\frac{R}{\delta}+1\right) +  \frac{R}{R+\delta} \left(1-\frac{R}{z}\right)&\mbox{for }z\ge R.
\end{cases}
\]
By construction, $c_{\delta,R}$ is a continuously differentiable, bounded, and strictly concave function on $\R_+$. 
We shall use the abbreviation
\[
\D_{\delta,R}(\eta) := \D_{c_{\delta,R}}(\eta).
\]
At any time $t$, let $\pi_{\opt}(t)$ and $\varphi_{\opt}(t,\tacka)$ denote the optimal transport plan and the Kantorovich--Rubinstein potential corresponding to $\D_{\delta,R}(\eta(t,\tacka))$. The Lipschitz condition in \eqref{3} can be rephrased as
\[
|\varphi_{\opt}(t,x) - \varphi_{\opt}(t,y)|\le \begin{cases} \log\left(\frac{|x-y|}{\delta}+1\right)& \mbox{for }|x-y|\le R,\\ 
\log\left(\frac{R}{\delta}+1\right) +  \frac{R}{R+\delta} \left(1-\frac{R}{|x-y|}\right)&\mbox{for }|x-y|\ge R.
\end{cases}.
\]
In particular, upon adding a constant to $\varphi_{\opt}$, we can always assume that $\|\varphi_{\opt}\|_{L^{\infty}}\le \log(R\delta^{-1} +1) + R(R+\delta)^{-1}$. Moreover,
because $c_{\delta,R}$ is Lipschitz (on $\R_+$), so is $\varphi_{\opt}(t,\tacka)$ with $
\|\grad\varphi_{\opt}\|_{L^{\infty}} \le \delta^{-1}$.
For further reference, $\varphi_{\opt} \in L^{\infty}(\R; W^{1,\infty}(\R^d))$ with
\begin{equation}
\label{12}
 \|\varphi_{\opt}\|_{L^{\infty}(L^{\infty})} \le \log\left(\frac{R}{\delta}+1\right)  + \frac{R}{R+\delta},\qquad \|\grad\varphi_{\opt}\|_{L^{\infty}(L^{\infty})} \le \frac1{\delta}.
\end{equation}
We also infer from \eqref{5} that $\grad\varphi_{\opt}(t,x) =\grad\varphi_{\opt}(t,y)$ with
\begin{equation}
\label{9}
\grad\varphi_{\opt}(t,x) =\begin{cases} \frac1{\delta+|x-y|} \frac{x-y}{|x-y|} &\mbox{if }|x-y|\le R,\vspace{.5em}\\  \frac{R^2}{R+\delta} \frac1{|x-y|^2}\frac{x-y}{|x-y|}&\mbox{if }|x-y|\ge R,\end{cases}
\end{equation}
for $d\pi_{\opt}(t)$-almost all $(x,y)$.

The heart of this paper is the following stability estimate for the Kantorovich--Rubinstein distance: 
\begin{prop}
\label{P1}
Let $1\le p,q\le \infty$ be given with $1/p+1/q=1$. For $u_1, u_2\in L^1((0,T); W^{1,p}(\R^d))$, $f_1, f_2\in L^1 (0,T);L^1\cap L^q(\R^d))$ and $\bar \rho_1, \bar\rho_2\in L^1 \cap L^q(\R^d)$, let $\rho_1$ and $\rho_2$ be corresponding solutions to the continuity equation \eqref{23}.
Define
\[
\eta: = \rho_1-\rho_2 - \left(\bar \rho_1-\bar\rho_2\right) -\int_0^t \left(f_1-f_2\right)\, ds.
\]
Then there exist   positive constants  $C_1$  and $C_2$ which are independent of $\delta$ and $R$ 
such that
\begin{equation}
\label{13}
\|\D_{\delta,R}(\eta)\|_{L^{\infty}}
\le C_1 \psi_p(\delta) + \frac{C_2}{\delta}\left(\|\bar\rho_1 - \bar\rho_2\|_{L^q} + \|u_1-u_2\|_{L^1(L^p)}  + \|f_1-f_2\|_{L^1(L^q)}\right),
\end{equation}
for any positive $\delta$ and $R$, where $\psi_p=1$ if $p>1$ and otherwise, $\psi_1$ is a continuous function on $\R_+$ with $\psi_1(\delta)/|\log\delta|\to0$ as $\delta\to0$.
\end{prop}

The proof of this proposition requires some preparation. We first compute the temporal rate of change of the Kantorovich--Rubinstein distance for solutions of the continuity equation. 

\begin{lemma}\label{L1}
Let  $j$ be a vector field in $L^1((0,T);L^1( \R^d))$ and $\eta$ be a mean-zero function in $L^1((0,T);L^1(\R^d))$ that satisfy the continuity equation 
\[
\partial_t\eta +\div j = 0
\]
distributionally in $(0,T)\times \R^d$.
Let $\varphi_{\opt}$ be a Kantorovich--Rubinstein potential corresponding to $\D_{\delta,R}(\eta)$ for some $\delta>0$ and $R>0$. Then $t\mapsto \D_{\delta,R}(\eta(t,\tacka))$ is weakly differentiable with
\begin{equation}
\label{7}
\frac{d}{dt} \D_{\delta,R}(\eta) = \int j \cdot \grad\varphi_{\opt} \, dx .
\end{equation}
\end{lemma}

\begin{proof} %Recall that we can always choose $\varphi_{\opt}$ to be bounded, thus $\varphi_{\opt}\in L^{\infty}(\R;W^{1,\infty}(\Omega))$. 
%It is clear that the mean of distributional solutions is preserved, and thus, the hypothesis that $\eta$ has mean zero is satisfies for all times if it is satisfied initially. In particular, $t\mapsto \D_{\delta,R}(\eta(t,\tacka))$ is well-defined.

We first notice that it is enough to show that \eqref{7} holds in the sense of distributions. Indeed, the right-hand side of \eqref{7} is bounded by
$\|\grad \varphi_{\opt}\|_{L^{\infty}}\|j\|_{L^1} $,
which is  integrable in time. By decomposing testfunctions into positive and negative parts and a standard approximation procedure, it is furthermore sufficient to prove
\begin{equation}
\label{8}
\int \frac{d\psi}{dt} \D_{\delta,R}(\eta) \,dt+ \iint  \psi  j\cdot \grad\varphi_{\opt} \, dxdt=0,
\end{equation}
for all {\em nonnegative} testfunctions $\psi\in C_c^{\infty}(0,T)$.

For notational convenience, for any time $t\in\R$, we will denote by $\varphi_t$ the  potential corresponding to the Kantorovich--Rubinstein distance $\D_{\delta,R}(\eta_t)$, i.e., $\varphi_t = \varphi_{\opt}(t,\tacka)$, where accordingly $\eta_t = \eta(t,\tacka)$. The functions $j_t$,  and $\psi_t$ are analogously defined.

By optimality in \eqref{3}, for any $h\in\R$ it holds that
\[
\D_{\delta,R}(\eta_t) - \D_{\delta,R}(\eta_{t-h}) \le \int \varphi_t(\eta_t-\eta_{t-h})\, dx.
\]
Integration against $\psi$ and a change of variables yield
\[
\int(\psi_t-\psi_{t+h})\D_{\delta,R}(\eta_t)\, dt \le \iint (\psi_t\varphi_t - \psi_{t+h}\varphi_{t+h})\eta_t\, dxdt.
\]
In order to appeal to the distributional formulation of the continuity equation, we shall approximate $\varphi$ by smooth functions $\varphi^{\eps}$ that are compactly supported in $\R^d$. This is possible because $\varphi$ is  uniformly bounded. We thus have
\begin{eqnarray*}
\lefteqn{\int (\psi_t - \psi_{t+h})\D_{\delta,R}(\eta_t)\, dt}\\
 &\le& \iint \frac{\partial}{\partial t}\left( \int_{t+h}^t \psi_s \varphi_s^{\eps}\, ds\right)\eta_t\, dxdt + o(1)\\
&=& - \iint j_t\cdot\left( \int_{t+h}^t \psi_s \grad \varphi_s^{\eps}\, ds\right)    dxdt + o(1)
\end{eqnarray*}
as $\eps \to0$. We recall that $\varphi_s$ is a Lipschitz function, so that we can undo the approximation. We may thus drop the $\eps$ in the above estimate. Dividing by $h$ and using Lebesgue's differentiation and dominated convergence theorems, we deduce the statement in \eqref{8} as $h\to0$.
\end{proof}

Using the marginal condition \eqref{2} and the calculation \eqref{5}, we can estimate the rate of change of $\D_{\delta,R}(\eta)$ in \eqref{7}.

\begin{lemma}\label{L2}
Let $1\le p,q\le \infty$ be given with $1/p+1/q=1$. 
Let $\eta \in L^1 \cap L^q(\R^d)$ be a function with zero mean and let $\pi_{\opt}$ and $\varphi_{\opt}$ be, respectively, a Kantorovich--Rubinstein potential and the optimal transport plan corresponding to $\D_{\delta,R}(\eta) = \D_{\delta,R}(\eta_+,\eta_-)$ for some $\delta>0$. Then, for any function $u$ in $ L^p(\R^d)$,
\begin{equation}
\label{8a}
\left| \int u\cdot \grad \varphi_{\opt}\eta\, dx\right|\le \iint \frac{|u(x) - u(y)|}{\delta + |x-y|}\, d\pi_{\opt}(x,y).
\end{equation}
\end{lemma}

\begin{proof}From the marginal condition \eqref{2} on transport plans we deduce that
\[
\int u\cdot \grad\varphi_{\opt}\eta\, dx = \iint u(x)\cdot \grad \varphi_{\opt}(x) - u(y)\cdot\grad\varphi_{\opt}(y)\, d\pi_{\opt}(x,y).
\]
We now apply the formula for the gradients of the potentials, \eqref{9}, to the effect that
\begin{eqnarray*}
\int u\cdot \grad\varphi_{\opt}\eta\, dx & = &\iint_{|x-y|\le R} \frac{u(x)-u(y)}{\delta + |x-y|}\cdot \frac{x-y}{|x-y|}\, d\pi_{\opt}(x,y)\\
&&\mbox{}+ \frac{R^2}{R+\delta} \iint_{|x-y|> R} \frac{u(x)-u(y)}{|x-y|^2}\cdot \frac{x-y}{|x-y|}\, d\pi_{\opt}(x,y).
\end{eqnarray*}
Now the statement of the lemma follows easily from the fact that
$
z\mapsto z^2(z+\delta)^{-1}
$
is an increasing function.
\end{proof}

We will first estimate the integral over the difference quotient in \eqref{8a} in the case where $u$ is a Sobolev function with $\grad u\in L^p$ for some $p>1$. Our proof uses one of the central tools from Calder\'on--Zygmund theory, namely the  (Hardy--Littlewood) maximal function operator $M$. The maximal function $Mf$ of a measurable function $f: \R^d\to\R$ is given by
\[
Mf(x)  = \sup_{r>0} \frac1{|B_r(x)|}\int_{B_r(x)} |f(y)|\, dy,
\]
for $x\in \R^d$. The operator is continuous from $L^p$ to $L^p$ if $p\in(1,\infty]$, thus
\begin{equation}
\label{10}
\|M f\|_{L^p} \le C \|f\|_{L^p},
\end{equation}
and bounds difference quotients in the sense that
\begin{equation}
\label{11}
\frac{|f(x) - f(y)|}{|x-y|} \le C \left( M |\grad f|(x) + M|\grad f|(y)\right),
\end{equation}
for almost all $x,y\in \R^d$. The first estimate if proved in \cite[p.\ 5, Theorem 1]{Stein70}, and the second one is a Morrey-type estimate and for instance contained in the proof of \cite[p.\ 143, Theorem 3]{EvansGariepy92}.

The idea of using maximal functions to control the right-hand side of \eqref{8a} by $\|\grad u\|_{L^p}$ goes back to the work of Crippa and De Lellis \cite{CrippaDeLellis08}. In the context of Kantorovich--Rubinstein distances, the techniques were previously used in \cite{BOS,Seis13b}.

\begin{lemma}\label{L3}
Let $p\in (1,\infty]$ and $q\in[1,\infty)$ be given with $1/p+1/q=1$. Let $\eta$ be a function in $L^1\cap L^q(\R^d)$ with zero mean and let $\pi$ be a transport plan in $\Pi(\eta_+,\eta_-)$. Then there exists a universal constant $C>0$ such that for any integrable  function $u$ with $\grad u\in L^p(\R^d)$,
\[
\iint \frac{|u(x)-u(y)|}{|x-y|}\, d\pi(x,y) \le C \|\eta\|_{L^q} \|\grad u\|_{L^p}.
\]
\end{lemma}

\begin{proof}
The statement is trivial if $p=\infty$, and we thus restrict to the case $p\in (1,\infty)$.

An application of the Morrey-type estimate \eqref{11} yields that
\[
\iint \frac{|u(x)-u(y)|}{|x-y|}\, d\pi(x,y) \le C\iint M |\grad u |(x) + M|\grad u|(y)\, d\pi(x,y).
\]
In the integrand on the right-hand side, the terms involving $x$ and $y$ are now separated. We can thus 
apply the marginal condition \eqref{2}, use $|\eta| = \eta_++\eta_-$, and obtain via H\"older's inequality
\[
\iint \frac{|u(x)-u(y)|}{|x-y|}\, d\pi(x,y) \le C \int M|\grad u|( \eta_+ + \eta_-)\, dx \le C \|\eta\|_{L^{q}} \|M|\grad u|\|_{L^p}.
\]
We deduce the statement of the lemma with the help of  \eqref{10}.
\end{proof}

\label{page1}
In the case where $u\in W^{1,1}(\R^d)$, the above argumentation breaks down, because the maximal function operator ceases to be continuous on $L^1$, cf.\ \eqref{10}. A way to overcome this difficulty was suggested by Jabin in \cite{Jabin10}: Because $\grad u$ belongs to $ L^1(\R^d)$, by the Dunford--Pettis theorem, there exists a nonnegative  continuous function $e$ on $\R_+$ with
\begin{equation}
\label{17}
\frac{e(\xi)}{\xi}\mbox{ increasing}, \quad \lim_{\xi\to \infty}\frac{e(\xi)}{\xi}=\infty,
\end{equation}
and
\begin{equation}
\label{18}
 \int e(|\grad u|)\, dx <\infty.
\end{equation}

With this function $e$ fixed, we can modify the result of the previous lemma as follows:

\begin{lemma}\label{L5}
 Let $\eta$ be a bounded mean-zero function in $L^1(\R^d)$ and let $\pi$ be a transport plan in $\Pi(\eta_+,\eta_-)$. Then there exists a universal constant $C>0$ such that for any function $u$ in $W^{1,1}(\R^d)$, there exists a continuous function $\psi$ depending only on $e$ and with $\psi(\xi)/|\log\xi|\to 0$ as $\xi\to 0$ such that
\[
\iint \frac{|u(x)-u(y)|}{\delta+ |x-y|}\, d\pi(x,y) \le C\psi(\delta) \left( \|\eta\|_{L^1} + \|\eta\|_{L^{\infty}} \|e(|\grad u|) \|_{L^1}\right).
\]
\end{lemma}

In Example \ref{E1} below, we will see that we cannot expect uniform bounds in $\delta$ if $u$ has only $BV$ regularity.

In our proof of Lemma \ref{L5}, we essentially imitate Jabin's \cite{Jabin10} modification of Crippa's and De Lellis's estimate \eqref{16e}.

\begin{proof}
We denote by $B(x,y)$ the ball of radius $|x-y|/2$ and center $(x+y)/2$. With the help of the elementary inequality,
\[
|u(x)-u(y)|\le C \int_{B(x,y)} \left(\frac1{|x-z|^{d-1}} +\frac1{|y-z|^{d-1}}\right)|\grad u(z)|\, dz,
\]
cf.\ \cite[Lemma 3.1]{Jabin10}, we estimate
\begin{eqnarray*}
\lefteqn{\iint \frac{|u(x)-u(y)|}{\delta+|x-y|} \, d\pi}\\
& \le& C \iint \int_{B(x,y)} \left(\frac1{|x-z|^{d-1}} +\frac1{|y-z|^{d-1}}\right)\frac{|\grad u(z)|}{\delta +|x-y|}\, dzd\pi.
\end{eqnarray*}
We let $M>0$ be an arbitrary constant and denote by $B_M(x,y)$ the subset of $B(x,y)$ in which $|\grad u|$ is bounded by $M$. As usual, the complementary set will be denoted by $B_M(x,y)^c$. Let $e$ be a nonnegative  continuous function with \eqref{17} and \eqref{18}.
On the one hand, because
\[
\int_{B(x,y)}  \left(\frac1{|x-z|^{d-1}} +\frac1{|y-z|^{d-1}}\right) \, dz \le C |x-y|,
\]
we have the estimate
\begin{eqnarray*}
\lefteqn{\iint \int_{B_M(x,y)}  \left(\frac1{|x-z|^{d-1}} +\frac1{|y-z|^{d-1}}\right)\frac{|\grad u(z)|}{\delta+|x-y|}\, dzd\pi}\\
&\le & M \iint \int_{B(x,y)}  \left(\frac1{|x-z|^{d-1}} +\frac1{|y-z|^{d-1}}\right)\frac{1}{\delta+|x-y|}\, dzd\pi\\
&\le &CM \|\eta\|_{L^1}.
\end{eqnarray*}
On the other hand, because $\xi\mapsto e(\xi)/\xi$ is increasing, it holds
\begin{eqnarray*}
\lefteqn{\iint \int_{B_M(x,y)^c}  \left(\frac1{|x-z|^{d-1}} +\frac1{|y-z|^{d-1}}\right)\frac{|\grad u(z)|}{\delta+|x-y|}\, dzd\pi}\\
&\le& \frac{M}{e(M)} \iint \int_{B(x,y)}  \left(\frac1{|x-z|^{d-1}} +\frac1{|y-z|^{d-1}}\right)\frac{e(|\grad u(z)|)}{\delta+|x-y|}\, dzd\pi.
\end{eqnarray*}
The integrand increases if we replace $|x-y|$ by $|x-z|$ or $|y-z|$ in the denominator and extend the inner integral over all of $\R^d$. We thus achieve that the integrand splits into a term depending only on $x$ and one depending only on $y$. Invoking the marginal condition \eqref{2} and using Fubini, we thus arrive at
\begin{eqnarray*}
\lefteqn{\iint \int_{B_M(x,y)}  \left(\frac1{|x-z|^{d-1}} +\frac1{|y-z|^{d-1}}\right)\frac{|\grad u(z)|}{\delta+|x-y|}\, dzd\pi}\\
&\le& \frac{M}{e(M)} \int e(|\grad u(z)|) \int \frac{|\eta(x)|}{|x-z|^{d-1}(\delta+|x-z|)}\, dx dz\\
&\le &C \frac{M}{e(M)}\left(|\log \delta| +1\right)\|\eta\|_{L^{\infty}} \| e(|\grad u|)\|_{L^1}.
\end{eqnarray*}
Combining the estimates on $B_M(x,y)$ and $B_M(x,y)^c$ and optimizing in $M $ yields the desired result with
\[
\psi(\delta) :=\inf_{M>0} \left(M + \frac{M}{e(M)} \left(|\log\delta|+1\right)\right).
\]
\end{proof}

\begin{ex}\label{E1}
The following construction shows that we cannot expect that Lemma \ref{L3} extends to the $BV$ case. In fact, we prove that there exists a vector field $u$ with $| u|_{BV}\sim 1$ and a mean zero function $\eta$ such that
\begin{equation}
\label{18a}
\iint \frac{|u(x)-u(y)|}{\delta+ |x-y|}\, d\pi(x,y)\sim |\log\delta| 
\end{equation}
as $\delta \ll 1$. For convenience, we consider the periodic one-dimensional setting.  For $x\in[0,1)$, we set
\[
\eta(x) = \begin{cases}  1 & \mbox{if }x\in\left[0,\frac12\right) \\ -1 & \mbox{if }x\in \left[\frac12,1\right), \end{cases}
\]
and extend $\eta$ periodically. Then the optimal transport plan $\pi_{\opt}$ is of the form
\[
\pi_{\opt} = (\id\times T)_{\#} \eta_+,
\]
i.e., the push-forward of $\eta_+$ by the map $\id\times T$, where $T$ is the optimal transport map given by
\[
T(x) = \begin{cases}   - x &\mbox{if } x \in\left(0 , \frac14\right)\\
1 - x &\mbox{if } x \in\left(\frac14,\frac12\right),
\end{cases}
\]
and extended periodically. If $u=\eta$, then $| u|_{BV}\sim 1$ and $u(x)-u(y) = 2 $ for $d\pi_{\opt}$-almost all $(x,y)$. Thus, if we denote by $|\tacka|_{\per}$ the periodic distance on the periodic interval $[0,1)_{\per}$, we have
\begin{eqnarray*}
\iint \frac{|u(x)-u(y)|}{\delta+|x-y|_{\per}} \, d\pi_{\opt}(x,y) &=& \int_{[0,1)_{\per}} \frac{2}{\delta +|x-T(x)|_{\per}} \eta_+(x)\, dx\\
&=& \int_0^{1/4} \frac2{\delta + 2x}\, dx + \int_{1/4}^{1/2} \frac2{\delta + 1-2x}\, dx\\
&=& 2 \log\left(\frac1{2\delta} +1\right).
\end{eqnarray*}
This proves \eqref{18a} if $\delta\ll1$.
\end{ex}

We are now in the position to proof Proposition \ref{P1}:

\begin{proof}[Proof of Proposition  \ref{P1}]
Notice that $\eta$ is constructed in such a way that its mean is zero for almost all times. Indeed, by an approximation argument, we verify that 
\[
\int \rho_i\, dx = \int \bar \rho_i\, dx + \int_0^t\int f_i\,dxds
\]
for $i=1,2$. We may change $\eta$ on a set of Lebesgue measure zero to achieve that its mean is constantly zero in time. Notice also that $\eta $ vanishes initially. The continuity equation satisfied by $\eta$ is of the form
\[
\partial_t \eta  + \div j = 0,
\]
where
\[
j := u_1 \eta + (u_1-u_2)\rho_2 + u_1(\bar \rho_1-\bar\rho_2) + u_1\int_0^t (f_1-f_2)\, ds
\]
is a function in $ L^1((0,T); L^1(\R^d))$. We infer thus from Lemma \ref{L1} that the function $t\mapsto \D_{\delta,R}(\eta(t,\tacka))$ is weakly differentiable with derivative
\[
\frac{d}{dt} \D_{\delta,R}(\eta) = \int j\cdot\grad \varphi_{\opt}\, dx.
\]
We remark that $\D_{\delta,R}(\eta(t,\tacka))\to 0$ as $t\to0$, which follows from the facts that solutions to the continuity equation \eqref{23} approach their initial value weakly in $L^1$ and   Kantorovich--Rubinstein distances metrize weak convergence (cf.\ \cite[Theorem 7.12]{Villani03}). 
Integration in time thus yields
\begin{eqnarray}
\| \D_{\delta,R}(\eta)\|_{L^{\infty}} &\le & \int_0^T\left|\int u_1\cdot \grad\varphi_{\opt}\eta \, dx \right| dt + \int_0^T \int |u_1-u_2|| \grad\varphi_{\opt} ||\rho_2|\, dxdt\nonumber\\
&&\mbox{}+ \int_0^T\int |u_1|| \grad \varphi_{\opt}| \left(|\bar \rho_1-\bar\rho_2|+ \int_0^t |f_1-f_2|\, ds\right)dxdt\label{30}.
\end{eqnarray}
The first term on the right-hand side is estimated via Lemmas \ref{L2}--\ref{L5} to the effect that
\[
\int_0^T\left|\int u_1\cdot\grad\varphi_{\opt}\eta\, dx\right|dt\le C_1 \psi_p(\delta) ,
\]
where $\psi_p$ and $C_1$ are  as in the statement of the proposition.
% By the definition of $\eta$, we furthermore have
%\[
%\|\eta\|_{L^{\infty}(L^q)} \le \|\rho_1-\rho_2\|_{L^{\infty}(L^q)} + \|\bar\rho_1- \bar\rho_2\|_{L^q} +  \|f_1-f_2\|_{L^1(L^q)}.
%\]
The remaining terms on the right-hand side of \eqref{30} are controlled by
\begin{eqnarray*}
\lefteqn{\|\grad\varphi_{\opt}\|_{L^{\infty}} \left(\|u_1-u_2\|_{L^1(L^p)} \|\rho_2\|_{L^{\infty}(L^q)}  + \|u_1\|_{L^1(L^p)} \|\bar\rho_1-\bar\rho_2\|_{L^q}\right.}\\
&& \left. + \|u_1\|_{L^1(L^p)} \| f_1-f_2\|_{L^1(L^q)}\right).\hspace{10em}
\end{eqnarray*}
Invoking \eqref{12} yields the desired estimate.

\end{proof}

To prove Theorem \ref{T1}, we need an additional estimate. For that purpose we define for $R>0$,
\[
\D_R(\eta) : =  \inf_{\pi\in\Pi(\eta_+,\eta_-)} \iint \min\{|x-y|,R\}\, d\pi(x,y).
\]
It is clear that $\D_R(\eta)=0$ if and only if $\eta=0$.

We have:
\
\begin{lemma}\label{L4}
Let $\eta$ be a mean zero  function in $L^1(\R^d)$. Then for any positive $\eps,\, \delta$, and $R$,
\[
\D_{R}(\eta) \le \delta \exp\left(\frac{\D_{\delta,R}(\eta)}{\eps}\right)\|\eta\|_{L^{1}} + \eps R  + R \log^{-1}\left(\frac{R}{\delta}+1\right)\D_{\delta,R}(\eta).
\]
In particular, if 
there exists a continuous function $\psi$ on $\R_+$ with $\psi(\xi)/|\log\xi|\to 0$ as $\xi \to 0$, and
\[
\sup_{\delta, R>0}\frac{ \D_{\delta,R}(\eta)}{\psi(\delta)}<\infty,
\]
then $\eta = 0$.
\end{lemma}

\begin{proof}
We write  $\kappa: =\D_{\delta,R}(\eta)$ for abbreviation. Let $D_R$ be the set of all points $(x,y) $ in $ \R^d\times \R^d$ whose distance is at most $R$. We moreover define 
$K$ as the subset of $D_R $ where $c_{\delta,R}(|x-y|) \le \eps^{-1}\kappa$ for all $(x,y)$. Then, denoting by $K^c$ the complement set of $K$ in $D_R$, it follows that the optimal transport plan $\pi_{\opt}$  satisfies the bound
\[
\pi_{\opt}[K^c] \le \frac{\eps}{\kappa} \D_{\delta,R}(\eta) = \eps.
\]
On the one hand, by the definition of $K$ and because $c_{\delta,R}^{-1}(\xi) = \delta(\exp(\xi)-1)$ for $\xi \le \log(\delta^{-1}R+1)$, we have that
\[
\iint_K |x-y|\, d\pi_{\opt} \le\delta \exp\left(\frac{\kappa}{\eps}\right)  \pi_{\opt}[K] \le \delta \exp\left(\frac{\kappa}{\eps}\right)\|\eta\|_{L^{1}} .
\]
On the other hand, 
\[
\iint_{K^c} |x-y|\, d\pi_{\opt} \le R \pi_{\opt}[K^c] \le \eps R.
\]
Finally, away from the diagonal we have
\[
\iint_{D_R^c} \, d\pi_{\opt} \le \frac{\kappa}{c_{\delta,R}(R)} = \frac{\kappa}{ \log\left(\frac{R}{\delta} +1\right)}.
\]
Combining the previous estimates and optimizing over all $\pi \in \Pi(\eta_+,\eta_-)$ yields
\[
\D_{R}(\eta) \le \delta \exp\left(\frac{\kappa}{\eps}\right)\|\eta\|_{L^{1}}   + \eps R +\frac{R\kappa}{ \log\left(\frac{R}{\delta} +1\right)},
\]
which is the first statement of the lemma.

For the second statement, we let first $\delta\to0$ and then $\eps\to 0$ and  find $\D_{R}(\eta)=0$. Thus $\eta=0$.
\end{proof}

It remains to establish our main results.

\begin{proof}[Proof of Theorem \ref{T1}]
Given  two solutions $\rho_1$ and $\rho_2$  of the Cauchy problem \eqref{23}, we consider their difference $\eta := \rho_1-\rho_2$. Then $\eta$ satisfies the homogeneous equation with zero initial datum. Applying Proposition \ref{P1}, we then obtain for any positive $\delta$ and $R$ that
\[
\|\D_{\delta,R}(\eta) \|_{L^{\infty}}\le C\psi_p(\delta),
\]
where $\psi_p=1$ if $p>1$ and $\psi_1$ continuous with $\psi_1(\delta)/|\log\delta|\to0$ as $\delta\to0$. In particular, from Lemma \ref{L4} we infer that $\eta=0$. This shows uniqueness.
\end{proof}

\begin{proof}[Proof of Theorem \ref{T2}]
%Weak compactness of the approximating solutions $\{\rho_{\nu}\}_{\nu\in\N}$ follows from the assumptions on the data via \eqref{25a}. This is clear in the case $q>1$. Otherwise we argue as in \cite{DiPernaLions89}: Let $\{\bar \rho_{k}\}_{k\in\N}$ and $\{f_k\}_{k\in\N}$ be smooth and compactly supported sequences that converge to $\bar \rho$ and $f$ in $L^1$ and $L^1(L^1)$, respectively. 
%
%
%
% We can thus extract a weakly converging subsequence. Passing to the limit in the distributional formulation of the continuity equation, \eqref{25}, and applying the uniqueness result of Theorem \ref{T1}, we find that the full sequence $\{\rho_{\nu}\}_{\nu\in\N}$  converges to $\rho$ weakly in $L^{\infty}(L^q)$, and thus in $L^1_{\loc}$.
%
%That this convergence actually takes place  strongly in $L^1_{\loc}$ can be seen as follows: By uniqueness we easily verify that $(\rho_{\nu})_+$, i.e., the positive part of $\rho_{\nu}$, satisfies the continuity equation with source $(f_{\nu})_+$ and initial datum $(\bar \rho_{\nu})_+$. The analogous statement holds true for the negative part $(\rho_{\nu})_-$. Arguing as above, we then find that $\{(\rho_{\nu})_{\pm}\}_{\nu\in\N}$ converge to $\rho_{\pm}$ weakly in $L_{\loc}^1$. It remains to observe that $|\rho_{\nu}| = (\rho_{\nu})_+ + (\rho_{\nu})_-$, and thus weak convergence upgrades to strong convergence.

Boundedness of  $\eta$ in $L^{\infty}(L^1\cap L^q)$ is an immediate consequence of the assumptions on the data. We focus thus 
on the stability estimate.

Thanks to Proposition \ref{P1}, there exists a constant $C$ with the desired properties  such that for any positive $\delta$ and $R$ we have
\[
\|\D_{\delta, R}(\eta)\|_{L^{\infty}} \le C\left(1+\frac{r}{\delta}\right).
\]
From here on, we will drop the constants in the displayed formulas. From Lemma \ref{L4} with $R=1$ and $\delta = r$ we then obtain that
\[
\|\D_{1}(\eta)\|_{L^{\infty}} \lesssim   r \exp\left(\frac{1}{\eps}\right) + \eps + \frac1{\log\left(\frac1{r}+1\right)}
\]
for any $\eps>0$. We choose $\eps = |\log \sqrt{r} |^{-1}$, to the effect of
\[
\|\D_1(\eta) \|_{L^{\infty}} \lesssim  \sqrt{r} + \frac1{\log\left(\frac1{r}\right)} + \frac1{\log\left(\frac1{r} +1\right)}.
\]
Because $r\ll1$, the right-hand side is of order $|\log r|^{-1}$. It remains thus to observe that
\[
\D_1(\eta) \sim \|\eta\|_{W^{-1,1}},
\]
which follows from the Kantorovich--Rubinstein duality formula  \eqref{3}, because $d(x,y)$ $ := \min\{|x-y|,1\}$ defines a metric on $\R^d$.
\end{proof}

\section*{Acknowledgements}

The author thanks L.\ Ambrosio for bringing references \cite{Loeper06b} and \cite{Loeper06a} to his attention. He furthermore thanks C.\ De Lellis for mentioning the counterexample for strong stability results (Example \ref{example}). 

\bibliography{coarsening.bib}
\bibliographystyle{abbrv}
\end{document}